\documentclass{amsart}
\usepackage{amsfonts,graphicx,rotating,ctable}
\usepackage{amssymb}

\makeatletter
\@namedef{subjclassname@2020}{\textup{2020} Mathematics Subject Classification}
\makeatother

\newtheorem{theorem}{Theorem}[section]
\newtheorem{lemma}[theorem]{Lemma}
\newtheorem{corollary}[theorem]{Corollary}
\newtheorem{proposition}[theorem]{Proposition}

\theoremstyle{definition}
\newtheorem{definition}[theorem]{Definition}
\newtheorem{example}[theorem]{Example}

\theoremstyle{remark}
\newtheorem{remark}[theorem]{Remark}

\numberwithin{equation}{section}

\newcommand{\im}{\operatorname{im}}


\begin{document}

\title[Gr\"obner bases for oriented Grassmann manifolds $\widetilde G_{2^t,3}$]{Gr\"obner bases in the mod $2$ cohomology of oriented Grassmann manifolds $\widetilde G_{2^t,3}$}

\author{Uro\v s A.\ Colovi\'c}
\address{University of Belgrade,
  Faculty of mathematics,
  Studentski trg 16,
  Belgrade,
  Serbia}
\email{mm21033@alas.matf.bg.ac.rs}

\author{Branislav I.\ Prvulovi\'c}
\address{University of Belgrade,
  Faculty of mathematics,
  Studentski trg 16,
  Belgrade,
  Serbia}
\email{bane@matf.bg.ac.rs}

\subjclass[2020]{Primary 55R40, 13P10; Secondary 55s05}



\keywords{Grassmann manifolds, Gr\"obner bases, Steenrod squares}

\begin{abstract}
For $n$ a power of two, we give a complete description of the cohomology algebra $H^*(\widetilde G_{n,3};\mathbb Z_2)$ of the Grassmann manifold $\widetilde G_{n,3}$ of oriented $3$-planes in $\mathbb R^n$. We do this by finding a reduced Gr\"obner basis for an ideal closely related to this cohomology algebra. Using this Gr\"obner basis we also present an additive basis for $H^*(\widetilde G_{n,3};\mathbb Z_2)$.
\end{abstract}

\maketitle



\section{Introduction}
\label{intro}

There are several well-known descriptions of the mod $2$ cohomology algebra of the Grassmann manifold $G_{n,k}$ of $k$-dimensional subspaces in $\mathbb R^n$, one of them being the Borel description \cite{Borel} in terms of the Stiefel--Whitney classes of the tautological vector bundle over $G_{n,k}$. The Grassmann manifold $\widetilde G_{n,k}$ of oriented $k$-dimensional subspaces in $\mathbb R^n$ is the universal (two-fold) covering of $G_{n,k}$, with the covering map $p:\widetilde G_{n,k}\rightarrow G_{n,k}$, which "forgets" the orientation of a $k$-plane. The subalgebra $\im\big(p^*:H^*(G_{n,k};\mathbb Z_2)\rightarrow H^*(\widetilde G_{n,k};\mathbb Z_2)\big)$ of $H^*(\widetilde G_{n,k};\mathbb Z_2)$ is described as a quotient of the polynomial algebra $\mathbb Z_2[\widetilde w_2,\ldots,\widetilde w_k]$ by a well-known ideal, where $\widetilde w_2,\ldots,\widetilde w_k$ are the Stiefel--Whitney classes of the tautological vector bundle $\widetilde\gamma_{n,k}$ over $\widetilde G_{n,k}$. However, this subalgebra $\im p^*$ is strictly smaller than $H^*(\widetilde G_{n,k};\mathbb Z_2)$, and in general, there is no complete description of the whole algebra $H^*(\widetilde G_{n,k};\mathbb Z_2)$.

Recently there has been some significant interest in determining $H^*(\widetilde G_{n,k};\mathbb Z_2)$ (see e.g.\ \cite{BasuChakraborty,Fukaya,Korbas:ChRank,KorbasRusin:Palermo}). The complete calculation of this algebra for $k=2$ was given by Korba\v s and Rusin in \cite{KorbasRusin:Palermo}. In that paper the authors also cover the case $k=3$ and $6\leq n\leq 11$. In \cite{BasuChakraborty} Basu and Chakraborty use the Serre spectral sequence to give an almost complete description of this algebra for $k=3$ and $n$ close to a power of two. Their description in the case $k=3$ and $n=2^t$ reads as follows (\cite[Theorem B(1)]{BasuChakraborty}):
\[H^*(\widetilde G_{2^t,3};\mathbb Z_2)\cong\frac{\frac{\mathbb Z_2[\widetilde w_2,\widetilde w_3]}{(g_{2^t-2},g_{2^t-1})}\otimes_{\mathbb Z_2}\mathbb Z_2[a_{2^t-1}]}{(a_{2^t-1}^2-Pa_{2^t-1})}.\]
Here, the part $\frac{\mathbb Z_2[\widetilde w_2,\widetilde w_3]}{(g_{2^t-2},g_{2^t-1})}$ corresponds to $\im p^*$, the class $a_{2^t-1}\in H^{2^t-1}(\widetilde G_{2^t,3};\mathbb Z_2)$ is the only "indecomposable" outside $\im p^*$, and $P$ is some unknown polynomial in $\widetilde w_2$ and $\widetilde w_3$.

In this paper we prove that the polynomial $P$ from this description vanishes, i.e., $a_{2^t-1}^2=0$ in $H^*(\widetilde G_{2^t,3};\mathbb Z_2)$, which means that $H^*(\widetilde G_{2^t,3};\mathbb Z_2)$ splits as the tensor product of $\im p^*$ and the exterior algebra on $a_{2^t-1}$.

\begin{theorem}\label{thm1} Let $t\geq2$ be an integer. Then we have the following isomorphism of graded $\mathbb Z_2$-algebras:
\[H^*(\widetilde G_{2^t,3};\mathbb Z_2)\cong\frac{\mathbb Z_2[\widetilde w_2,\widetilde w_3]}{(g_{2^t-2},g_{2^t-1})}\otimes_{\mathbb Z_2}\Lambda_{\mathbb Z_2}(a_{2^t-1}),\]
where $\widetilde w_i$ corresponds to the $i$-th Stiefel--Whitney class of the tautological bundle $\widetilde\gamma_{2^t,3}$ over $\widetilde G_{2^t,3}$, and the degree of $a_{2^t-1}$ is $2^t-1$.
\end{theorem}

The key ingredient in the proof of Theorem \ref{thm1} will be a reduced Gr\"obner basis for the ideal $(g_{2^t-2},g_{2^t-1})$ in the polynomial ring $\mathbb Z_2[\widetilde w_2,\widetilde w_3]$. That basis was found by Fukaya in \cite{Fukaya}, but here we explicitly determine the polynomials from the basis and present a different proof that it is a Gr\"obner basis. This basis conside\-rably facilitates calculations in $H^*(\widetilde G_{2^t,3};\mathbb Z_2)$, as we demonstrate in the paper. In particular, it produces an additive basis for $H^*(\widetilde G_{2^t,3};\mathbb Z_2)$ (a basis of $H^*(\widetilde G_{2^t,3};\mathbb Z_2)$ considered as a vector space over $\mathbb Z_2$). Concretely, this additive basis consists of the monomials
$a_{2^t-1}^r\widetilde w_2^b\widetilde w_3^c$ such that $r<2$ and for every $i\in\{0,1,\ldots,t-1\}$ either $b<2^{t-1}-2^i$ or $c<2^i-1$.

The paper is divided into four parts (including this introductory section). In Section \ref{cohomology} we give some background on the mod $2$ cohomology of $G_{n,k}$ and $\widetilde G_{n,k}$. Section \ref{groebner} is devoted to obtaining the above mentioned Gr\"obner basis and deriving some useful consequences. Finally, in Section \ref{proof} we calculate some Steenrod squaring operations in $H^*(\widetilde G_{2^t,3};\mathbb Z_2)$ and prove Theorem \ref{thm1}.

Throughout the paper we deal with the mod $2$ cohomology only, so all cohomo\-logy groups will be understood to have $\mathbb Z_2$ coefficients.

\section{Cohomology algebras $H^*(G_{n,k})$ and $H^*(\widetilde G_{n,k})$}
\label{cohomology}

Probably the most notable description of the cohomology algebra $H^*(G_{n,k})$ is due to Borel \cite{Borel}:
\begin{equation}\label{kohomologijaneorijentisanog}
H^*(G_{n,k})\cong\frac{\mathbb Z_2[w_1,w_2,\ldots,w_k]}{(\overline w_{n-k+1},\ldots,\overline w_n)},
\end{equation}
where the polynomials $\overline w_r$ are obtained from the relation
\begin{equation}\label{rel}
1+\overline w_1+\overline w_2+\cdots=\frac{1}{1+w_1+w_2+\cdots+w_k},
\end{equation}
as the appropriate homogeneous parts of the power series. A routine calculation leads to the following explicit formula:
\begin{equation}\label{dualpol}
\overline w_r=\sum_{a_1+2a_2+\cdots+ka_k=r}[a_1,a_2,\ldots,a_k]\,w_1^{a_1}w_2^{a_2}\cdots w_k^{a_k},
\end{equation}
where $[a_1,a_2,\ldots,a_k]:=\binom{a_1+a_2+\cdots+a_k}{a_1}
\binom{a_2+\cdots+a_k}{a_2}\cdots\binom{a_{k-1}+a_k}{a_{k-1}}$ is the multinomial coefficient (considered modulo $2$). Also, from (\ref{rel}) one can easily obtain the recurrence formula:
\begin{equation}\label{recdualpol}
\overline w_{r+k}=w_1\overline w_{r+k-1}+w_2\overline w_{r+k-2}+\cdots+w_k\overline w_r.
\end{equation}

In the isomorphism (\ref{kohomologijaneorijentisanog}) $w_i$ from the right-hand side corresponds to the $i$-th Stiefel--Whitney class $w_i(\gamma_{n,k})\in H^i(G_{n,k})$ of the canonical ($k$-dimensional) vector bundle $\gamma_{n,k}$ over $G_{n,k}$.

The map $p:\widetilde G_{n,k}\rightarrow G_{n,k}$, which sends an oriented $k$-plane in $\mathbb R^n$ to that same plane without specified orientation, is known to be a two-fold covering map. It is also clear from definitions of the canonical vector bundles $\gamma_{n,k}$ and $\widetilde\gamma_{n,k}$ over $G_{n,k}$ and $\widetilde G_{n,k}$ respectively, that $p^*(\gamma_{n,k})=\widetilde\gamma_{n,k}$, so for the Stiefel--Whitney classes we have
\[w_i(\widetilde\gamma_{n,k})=w_i(p^*(\gamma_{n,k}))=p^*(w_i(\gamma_{n,k})), \qquad 1\leq i\leq k.\]

The Gysin exact sequence associated to the double covering $p$ is also known to be of the form:
\[\cdots\rightarrow H^{j-1}(G_{n,k})\xrightarrow{w_1}H^j(G_{n,k})\xrightarrow{p^*}H^j(\widetilde{G}_{n,k})\xrightarrow{} H^j(G_{n,k})\xrightarrow{w_1}\cdots,\]
where $H^{j-1}(G_{k,n})\xrightarrow{w_1}H^j(G_{k,n})$ is the multiplication with $w_1=w_1(\gamma_{n,k})$. From this sequence we conclude that
\begin{equation}\label{ker=im}
\ker\big(H^*(G_{n,k})\xrightarrow{p^*}H^*(\widetilde G_{n,k})\big)=\im\big(H^*(G_{n,k})\xrightarrow{w_1}H^*(G_{n,k})\big).
\end{equation}
In particular, $w_1(\widetilde\gamma_{n,k})=p^*(w_1)=0$.

Consider now the following composition of algebra morphisms:
\begin{equation}\label{kompozicija}
\mathbb Z_2[\widetilde w_2,\ldots,\widetilde w_k]\hookrightarrow\mathbb Z_2[w_1,w_2,\ldots,w_k]\twoheadrightarrow H^*(G_{n,k})\xrightarrow{p^*}H^*(\widetilde G_{n,k}),
\end{equation}
where the first map sends $\widetilde w_i$ to $w_i$, and the second is the quotient map coming from (\ref{kohomologijaneorijentisanog}). It is clear that this composition maps $\widetilde w_i$ to $w_i(\widetilde\gamma_{n,k})$, $2\leq i\leq k$, and since $w_1(\widetilde\gamma_{n,k})=0$, the image of the composition is equal to $\im p^*$. A routine calculation (which uses (\ref{ker=im}) and (\ref{kohomologijaneorijentisanog})) shows that the kernel of the composition (\ref{kompozicija}) is equal to the ideal $(g_{n-k+1},\ldots,g_n)$ generated by the modulo $w_1$ reductions $g_r$ of the polynomials $\overline w_r$, $n-k+1\leq r\leq n$. Therefore, we obtain the isomorphism:
\begin{equation}\label{kohomologijaorijentisanog}
\im p^*\cong\frac{\mathbb Z_2[\widetilde w_2,\ldots,\widetilde w_k]}{(g_{n-k+1},\ldots,g_n)},
\end{equation}
where $\widetilde w_i$ (from the right-hand side) corresponds to $w_i(\widetilde\gamma_{n,k})\in\im p^*$. As usual, we shall denote these Stiefel--Whitney classes by $\widetilde w_i$ as well, and it will be clear from the context whether $\widetilde w_i$ is a variable in the polynomial algebra $\mathbb Z_2[\widetilde w_2,\ldots,\widetilde w_k]$ or a cohomology class in $H^*(\widetilde G_{n,k})$.

Since the polynomials $g_r\in\mathbb Z_2[\widetilde w_2,\ldots,\widetilde w_k]$ are the modulo $w_1$ reductions of the polynomials $\overline w_r\in\mathbb Z_2[w_1,w_2,\ldots,w_k]$ (i.e., $g_r$ is obtained from $\overline w_r$ by simply deleting all monomials which contain the variable $w_1$, and putting $\sim$ over all $w$'s), they are easily calculated from formula (\ref{dualpol}):
\[g_r=\sum_{2a_2+\cdots+ka_k=r}\binom{a_2+\cdots+a_k}{a_2}\cdots\binom{a_{k-1}+a_k}{a_{k-1}}\,\widetilde w_2^{a_2}\cdots\widetilde w_k^{a_k}.\]
Reducing the recurrence formula (\ref{recdualpol}) modulo $w_1$ leads to the relation
\[g_{r+k}=\widetilde w_2g_{r+k-2}+\cdots+\widetilde w_kg_r.\]

\bigskip

In the rest of the paper we confine ourselves to the case $k=3$ and $n=2^t$ (for some integer $t\geq2$). In this case the preceding two formulas simplify to:
\begin{equation}\label{gpolk3}
g_r=\sum_{2b+3c=r}{b+c\choose b}\,\widetilde w_2^b\widetilde w_3^c,
\end{equation}
and
\begin{equation}\label{recgpolk3}
g_{r+3}=\widetilde w_2g_{r+1}+\widetilde w_3g_r.
\end{equation}

\begin{example}\label{exg}
Let us now use (\ref{gpolk3}) and (\ref{recgpolk3}) to calculate a few of these polynomials:
\begin{align*}
g_1&=\sum_{2b+3c=1}\mbox{$\binom{b+c}{b}$}\widetilde w_2^b\widetilde w_3^c=0;\\
g_2&=\sum_{2b+3c=2}\mbox{$\binom{b+c}{b}$}\widetilde w_2^b\widetilde w_3^c=\widetilde w_2;\\
g_3&=\sum_{2b+3c=3}\mbox{$\binom{b+c}{b}$}\widetilde w_2^b\widetilde w_3^c=\widetilde w_3;\\
g_4&=\widetilde w_2g_2+\widetilde w_3g_1=\widetilde w_2^2;\\
g_5&=\widetilde w_2g_3+\widetilde w_3g_2=0;\\
g_6&=\widetilde w_2g_4+\widetilde w_3g_3=\widetilde w_2^3+\widetilde w_3^2;\\
g_7&=\widetilde w_2g_5+\widetilde w_3g_4=\widetilde w_2^2\widetilde w_3;\\
g_8&=\widetilde w_2g_6+\widetilde w_3g_5=\widetilde w_2^4+\widetilde w_2\widetilde w_3^2;\\
g_9&=\widetilde w_2g_7+\widetilde w_3g_6=\widetilde w_3^3.
\end{align*}
\end{example}

The ideal in (\ref{kohomologijaorijentisanog}) is now $(g_{2^t-2},g_{2^t-1},g_{2^t})$. However, from (\ref{recgpolk3}) we have $g_{2^t}=\widetilde w_2g_{2^t-2}+\widetilde w_3g_{2^t-3}$, and from \cite[Lemma 2.3(i)]{Korbas:ChRank} we know that $g_{2^t-3}=0$, which implies that $g_{2^t}=\widetilde w_2g_{2^t-2}$, leading to the conclusion that the polynomial $g_{2^t}$ is redundant in the generating set of the ideal $(g_{2^t-2},g_{2^t-1},g_{2^t})$. Finally, we obtain that
\begin{equation}\label{imp*}
\im p^*\cong\frac{\mathbb Z_2[\widetilde w_2,\widetilde w_3]}{(g_{2^t-2},g_{2^t-1})}.
\end{equation}

In the following section we exhibit a Gr\"obner basis for the ideal $(g_{2^t-2},g_{2^t-1})$, which will make calculations in $\im p^*$ (and consequently, in $H^*(\widetilde G_{2^t,3})$) much easier.

\begin{remark}
The quotient algebra $\mathbb Z_2[\widetilde w_2,\widetilde w_3]/(g_{2^t-2},g_{2^t-1})$ is isomorphic to $\im p^*$ in the case $n=2^t-1$ as well. Namely, the corresponding ideal for that case is $(g_{2^t-3},g_{2^t-2},g_{2^t-1})=(g_{2^t-2},g_{2^t-1})$, since $g_{2^t-3}=0$. The cohomology algebra $H^*(\widetilde G_{2^t-1,3})$ was studied in \cite{Fukaya} and a Gr\"obner basis for the ideal $(g_{2^t-2},g_{2^t-1})$ was obtained there. In this paper we identify that basis as a subset of $\{g_r\mid r\ge2^t-2\}$ and give a different proof that it is a Gr\"obner basis.
\end{remark}

\section{Gr\"obner bases}
\label{groebner}

\subsection{Background on Gr\"obner bases}

We begin with a very brief overview of the basic theory of Gr\"obner bases over a field. For more details the reader is referred to \cite[Chapter 5]{Becker} (although here we use a somewhat different notation and terminology than the ones used in \cite{Becker}).

Let $\mathbb K$ be a field and $\mathbb K[\underline X]$ the polynomial algebra over $\mathbb K$ in some (finite) number of variables (jointly denoted by $\underline X$). The set of all monomials in $\mathbb K[\underline X]$ will be denoted by $M$. Let $\preceq$ be a well ordering of $M$ (a total ordering such that every nonempty subset of $M$ has a least element) with the property that $m_1\preceq m_2$ implies $mm_1\preceq mm_2$, for all $m,m_1,m_2\in M$.

For a polynomial $p=\sum_{i=1}^r\alpha_i m_i\in\mathbb K[\underline X]$, where $m_i\in M$ are pairwise different and $\alpha_i\in\mathbb K\setminus\{0\}$, let $M(p):=\{m_i\mid1\leq i\leq r\}$. If $p\neq0$, i.e., $M(p)\neq\emptyset$, we define the {\it leading monomial of} $p$, denoted by $\mathrm{LM}(p)$, as $\max M(p)$ with respect to $\preceq$. The {\it leading coefficient of} $p$, denoted by $\mathrm{LC}(p)$, is the coefficient of $\mathrm{LM}(p)$ in $p$ and the {\it leading term of} $p$ is $\mathrm{LT}(p):=\mathrm{LC}(p)\cdot\mathrm{LM}(p)$. (In what follows we will work with the two-element field $\mathbb Z_2$, and in that case $\mathrm{LT}(p)=\mathrm{LM}(p)$.)

There is a number of equivalent ways for defining Gr\"obner bases. We choose one of them and then list some equivalent conditions.

\begin{definition}\label{grebner}
Let $F\subset\mathbb K[\underline X]$ be a finite set of nonzero polynomials and $I$ the ideal in $\mathbb K[\underline X]$ generated by $F$. We say that $F$ is a {\it Gr\"obner basis} for $I$ (with respect to $\preceq$) if for each $p\in I\setminus\{0\}$ there exists $f\in F$ such that $\mathrm{LM}(f)\mid\mathrm{LM}(p)$.
\end{definition}

In order to formulate some equivalent conditions we need the notion of reduction of polynomials. As in the definition, let $F\subset\mathbb K[\underline X]$ be a finite set of nonzero polynomials. For polynomials $p,q\in\mathbb K[\underline X]$ we say that $p$ \em reduces to $q$ modulo $F$ \em if there exist $n\geq1$ and polynomials $p_1,p_2,\ldots,p_n\in K[\underline X]$ such that $p_1=p$, $p_n=q$ and for every $i\in\{1,2,\ldots,n-1\}$
\[p_{i+1}=p_i-\frac{\mathrm{LC}(p_i)}{\mathrm{LC}(f_i)}\cdot m_i\cdot f_i,\]
for some $f_i\in F$ and $m_i\in M$ such that $m_i\cdot\mathrm{LM}(f_i)=\mathrm{LM}(p_i)$ (so that the leading term $\mathrm{LT}(p_i)$ cancels out on the right-hand side, and we get $\mathrm{LM}(p_{i+1})\prec\mathrm{LM}(p_i)$).

The proof of the following proposition can be found in \cite[Proposition 5.38]{Becker}.

\begin{proposition}\label{ekvusl}
Let $F\subset\mathbb K[\underline X]$ be a finite set of non-zero polynomials and $I$ the ideal in $\mathbb K[\underline X]$ generated by $F$. The following conditions are equivalent:
\begin{itemize}
\item[(1)] $F$ is a Gr\"obner basis for $I$;
\item[(2)] every $p\in I$ reduces to zero modulo $F$;
\item[(3)] the set of classes (cosets) of all monomials that are not divisible by any of the leading monomials $\mathrm{LM}(f)$, $f\in F$, is a basis of the quotient vector space $\mathbb K[\underline X]/I$.
\end{itemize}
\end{proposition}

For two nonzero polynomials $p,q\in\mathbb K[\underline X]$, the \em $S$-polynomial \em of $p$ and $q$ is defined as
\[S(p,q):=\mathrm{LC}(q)\cdot\frac{u}{\mathrm{LM}(p)}\cdot p-\mathrm{LC}(p)\cdot\frac{u}{\mathrm{LM}(q)}\cdot q,\] where $u=\mathrm{lcm}(\mathrm{LM}(p),\mathrm{LM}(q))$ is the least common multiple of $\mathrm{LM}(p)$ and $\mathrm{LM}(q)$. Note that on the right-hand side the term $\mathrm{LC}(q)\cdot\mathrm{LC}(p)\cdot u$ cancels out, so $S(p,q)$ is either zero or $\mathrm{LM}(S(p,q))\prec u$. Note also that $S(p,p)=0$ and $S(q,p)=-S(p,q)$.

Now we can formulate the Buchberger criterion \cite[Theorem 5.48]{Becker}.

\begin{theorem}\label{buchberger}
Let $F\subset\mathbb K[\underline X]$ be a finite set of nonzero polynomials and $I$ the ideal in $\mathbb K[\underline X]$ generated by $F$. Then $F$ is a Gr\"obner basis for $I$ if and only if $S(f_1,f_2)$ reduces to zero modulo $F$ for all $f_1,f_2\in F$.
\end{theorem}

For our purposes the notion of an $m$-representation (for a monomial $m\in M$) will play an important role. Let $F\subset\mathbb K[\underline X]$ be a finite set of nonzero polynomials, $m\in M$ a fixed monomial, and $p$ a nonzero polynomial in $\mathbb K[\underline X]$. If
\begin{equation}\label{mreprezentacija}
p=\sum_{i=1}^la_im_if_i,
\end{equation}
for some $a_i\in\mathbb K\setminus\{0\}$, $m_i\in M$ and $f_i\in F$ ($1\leq i\leq l$), such that $\mathrm{LM}(m_if_i)\preceq m$ for all $i\in\{1,\ldots,l\}$, then we say that (\ref{mreprezentacija}) is an $m$-{\it representation} of $p$ with respect to $F$. (Note that we do not require $f_i$-s to be pairwise different.)

Now we can formulate the theorem which will be essential in our calculation \cite[Theorem 5.64]{Becker}.

\begin{theorem}\label{mreprgrebner}
Let $F\subset\mathbb K[\underline X]$ be a finite set of nonzero polynomials and $I$ the ideal in $\mathbb K[\underline X]$ generated by $F$. If for all $f_1,f_2\in F$, $S(f_1,f_2)$ either equals zero or has an $m$-representation with respect to $F$ for some $m\prec\mathrm{lcm}(\mathrm{LM}(f_1),\mathrm{LM}(f_2))$, then $F$ is a Gr\"obner basis for $I$.
\end{theorem}

We conclude this overview with the notion of a reduced Gr\"obner basis. A Gr\"obner basis $F$ for an ideal $I$ is {\it reduced} if all polynomials in $F$ are monic ($\mathrm{LC}(f)=1$ for all $f\in F$) and there are no distinct polynomials $f_1,f_2\in F$ such that $\mathrm{LM}(f_1)$ divides some monomial from $M(f_2)$. It is a theorem \cite[Theorem 5.43]{Becker} that for every ideal $I\trianglelefteq\mathbb K[\underline X]$ there exists a unique reduced Gr\"obner basis (with respect to $\preceq$).

\subsection{Gr\"obner basis for the ideal $(g_{2^t-2},g_{2^t-1})\trianglelefteq\mathbb Z_2[\widetilde w_2,\widetilde w_3]$}

Let $t\geq2$ be a fixed integer and let $I$ be the ideal $(g_{2^t-2},g_{2^t-1})$ in the polynomial ring $\mathbb Z_2[\widetilde w_2,\widetilde w_3]$. We now want to determine the reduced Gr\"obner basis for $I$. We will use the lexicographic order on monomials in $\mathbb Z_2[\widetilde w_2,\widetilde w_3]$ with $\widetilde w_2\succ\widetilde w_3$. This means that
\[\widetilde w_2^{b_1}\widetilde w_3^{c_1}\preceq\widetilde w_2^{b_2}\widetilde w_3^{c_2} \qquad \Longleftrightarrow \qquad b_1<b_2 \quad \vee \quad (b_1=b_2 \quad \wedge \quad c_1\leq c_2).\]

For a nonnegative integer $i$, let $f_i$ be the polynomial $g_{2^t-3+2^i}$. Then we have $f_0=g_{2^t-2}$ and $f_1=g_{2^t-1}$. We already know that $I=(g_{2^t-2},g_{2^t-1})=(g_{2^t-2},g_{2^t-1},g_{2^t})$ (since $g_{2^t}=\widetilde w_2g_{2^t-2}$), and from (\ref{recgpolk3}) we conclude that $g_r\in I$ for all $r\geq2^t-2$. In particular, $f_i\in I$ for all $i\geq0$, so
\[I=(f_0,f_1)=(f_0,f_1,\ldots,f_{t-1}).\]
We are going to prove that $\{f_0,f_1,\ldots,f_{t-1}\}$ is the reduced Gr\"obner basis for $I$ (with respect to the lexicographic order $\preceq$). First, in order to determine the leading monomials of these polynomials, we prove an arithmetic lemma.

\begin{lemma}\label{lowerbound}
Let $b$, $c$ and $i$ be nonnegative integers such that:
\begin{enumerate}
    \item $2b+3c=2^is-3$ for some integer $s$;
    \item $\binom{b+c}{c}\equiv1\pmod2$.
\end{enumerate}
Then $c\ge2^i-1$.
\end{lemma}
\begin{proof}
The lemma is obvious for $i=0$. If $i\ge1$, note that condition (1) forces $c$ to be odd. In particular, $c\ge1$, so the lemma is true for $i=1$ as well. For $i\ge2$, let $b$ and $c$ be nonnegative integers such that (1) and (2) hold, and assume to the contrary that $c<2^i-1$. Since $c$ is odd, we have that $c=2d-1$ for some positive integer $d<2^{i-1}$. Then $2(b+c)+2d-1=2(b+c)+c=2b+3c=2^is-3$, so $b+c=2^{i-1}s-1-d$.
We will obtain a contradiction as soon as we prove
\[\binom{2^{i-1}s-1-d}{2d-1}\equiv0\pmod2.\]
We do this by Lucas' formula -- if $\sum_j\alpha_j2^j$ and $\sum_j\beta_j2^j$ are the binary expansions of $2^{i-1}s-1-d$ and $2d-1$ respectively, we are going to find a nonnegative integer $l$ such that $\alpha_l=0$ and $\beta_l=1$. Let $2^l$ ($l\geq0$) be the greatest power of two dividing $d$. Then $d=2^l(2k-1)$ for some positive integer $k$, and we know that $l\leq i-2$ (since $d<2^{i-1}$). Now we have
\begin{align*}
2^{i-1}s-1-d&=2^{i-1}s-1-2^l(2k-1)=2^l(2^{i-1-l}s-2k+1)-1\\
&=2^{l+1}(2^{i-2-l}s-k)+2^l-1=2^{l+1}(2^{i-2-l}s-k)+2^{l-1}+\cdots+1,
\end{align*}
and we conclude that $\alpha_l=0$. On the other hand,
\[2d-1=2^{l+1}(2k-1)-1=2^{l+2}(k-1)+2^{l+1}-1=2^{l+2}(k-1)+2^l+\cdots+1,\]
so $\beta_l=1$, completing the proof.
\end{proof}

\begin{proposition}\label{propvodecimonomi}
Let $i\in\{0,1,\ldots,t-1\}$. Then $f_i\neq0$ and
\[\mathrm{LM}(f_i)=\widetilde w_2^{2^{t-1}-2^i}\widetilde w_3^{2^i-1}.\]
Moreover, $f_{t-1}=\mathrm{LM}(f_{t-1})=\widetilde w_3^{2^{t-1}-1}$.
\end{proposition}
\begin{proof}
We have that
\[f_i=g_{2^t-3+2^i}=\sum_{2b+3c=2^t-3+2^i}{b+c\choose b}\widetilde w_2^b\widetilde w_3^c,\]
and so, we are looking for maximal $b$, or equivalently, minimal $c$, such that $2b+3c=2^t-3+2^i=2^i(2^{t-i}+1)-3$ and $\binom{b+c}{b}=\binom{b+c}{c}\equiv1\pmod2$.
According to Lemma \ref{lowerbound}, it is enough to prove that
$$\binom{2^{t-1}-2^i+2^i-1}{2^{i}-1}\equiv1\pmod2,$$
which is obvious from Lucas' formula.

As far as the second claim is concerned, if $\widetilde w_2^b\widetilde w_3^c$ is a monomial (with nonzero coefficient) in $f_{t-1}$, then $3c\leq2b+3c=2^t-3+2^{t-1}=3(2^{t-1}-1)$, i.e., $c\leq2^{t-1}-1$. By Lemma \ref{lowerbound}, $c\geq2^{t-1}-1$, and so, $f_{t-1}=\widetilde w_3^{2^{t-1}-1}$.
\end{proof}

Using induction on $i$, one can easily generalize identity (\ref{recgpolk3}), to obtain
\[g_{r+3\cdot2^i}=\widetilde w_2^{2^i}g_{r+2^i}+\widetilde w_3^{2^i}g_r, \quad i\ge0\]
(see \cite[(2.6)]{Korbas:ChRank}). Putting $r=2^t-3+2^i$, we get that for all nonnegative integers $i$,
\begin{equation}\label{lemspol1}
\widetilde w_3^{2^i}f_i+\widetilde w_2^{2^i}f_{i+1}=f_{i+2}.
\end{equation}

Our strategy is to use Theorem \ref{mreprgrebner}, so we need to work with $S$-polynomials of the polynomials $f_0,f_1,\ldots,f_{t-1}$. If $0\leq i\leq t-2$, note that, by Proposition \ref{propvodecimonomi}, the left-hand side in (\ref{lemspol1}) is actually the $S$-polynomial of $f_i$ and $f_{i+1}$. Therefore,
\begin{equation}\label{eqspol}
S(f_i,f_{i+1})=f_{i+2}, \qquad 0\leq i\leq t-2.
\end{equation}

The following lemma establishes crucial relations between $S$-polynomials of the polynomials $f_0,f_1,\ldots,f_{t-1}$.

\begin{lemma}\label{lemspol2}
For $0\leq i\leq j\leq t-2$ we have
\[S(f_i,f_{j+1})=\widetilde w_3^{2^j}S(f_i,f_j)+\widetilde w_2^{2^j-2^i}f_{j+2}.\]
\end{lemma}
\begin{proof}
By Proposition \ref{propvodecimonomi} we have
\[\mathrm{LM}(f_i)=\widetilde w_2^{2^{t-1}-2^i}\widetilde w_3^{2^i-1}, \quad \mathrm{LM}(f_j)=\widetilde w_2^{2^{t-1}-2^j}\widetilde w_3^{2^j-1},\]
from which it follows that
\[S(f_i,f_j)=\widetilde w_3^{2^j-2^i}f_i+\widetilde w_2^{2^j-2^i}f_j.\]
Now we use this and (\ref{lemspol1}) to calculate:
\begin{align*}
S(f_i,f_{j+1})=&\ \widetilde w_3^{2^{j+1}-2^i}f_i+\widetilde w_2^{2^{j+1}-2^i}f_{j+1}\\
=&\ \widetilde w_3^{2^{j+1}-2^i}f_i+\widetilde w_2^{2^j-2^i}\widetilde w_3^{2^j}f_j+\widetilde w_2^{2^j-2^i}\widetilde w_3^{2^j}f_j+\widetilde w_2^{2^{j+1}-2^i}f_{j+1}\\
=&\ \widetilde w_3^{2^j}(\widetilde w_3^{2^{j}-2^i}f_i+\widetilde w_2^{2^{j}-2^i}f_j)+\widetilde w_2^{2^{j}-2^i}(\widetilde w_3^{2^j}f_j+\widetilde w_2^{2^j}f_{j+1})\\
=&\ \widetilde w_3^{2^j}S(f_i,f_j)+\widetilde w_2^{2^{j}-2^i}f_{j+2},
\end{align*}
which is what we wanted to prove.
\end{proof}

We are now able to prove the reported result (cf.\ \cite[Theorem 4.7]{Fukaya}).

\begin{theorem}\label{grebnerimp*}
The set $\{f_0,f_1,\ldots,f_{t-1}\}$ is the reduced Gr\"obner basis for $I$ (with respect to the lexicographic order $\preceq$).
\end{theorem}
\begin{proof}
Let $i$ and $j$ be integers such that $0\le i<j\le t-1$. According to Theorem \ref{mreprgrebner}, it is sufficient to find an $m$-representation of $S(f_i,f_j)$ with respect to $\{f_0,f_1,\ldots,f_{t-1}\}$, for some monomial $m\prec\mathrm{lcm}(\mathrm{LM}(f_i),\mathrm{LM}(f_j))=\widetilde w_2^{2^{t-1}-2^i}\widetilde w_3^{2^j-1}$ (if $S(f_i,f_j)\neq0$).

We shall prove that
\begin{equation}\label{Sfifj}
S(f_i,f_j)=\sum_{k=i+2}^{j+1}\widetilde w_2^{2^{k-2}-2^i}\widetilde w_3^{2^j-2^{k-1}}f_k.
\end{equation}
Note that for $j=t-1$, $f_{j+1}=f_t=g_{2^{t+1}-3}=0$ by \cite[Lemma 2.3(i)]{Korbas:ChRank}, so for all nonzero summands in (\ref{Sfifj}) we have $k\leq t-1$. Since
\begin{align*}
\mathrm{LM}\Big(\widetilde w_2^{2^{k-2}-2^i}\widetilde w_3^{2^j-2^{k-1}}f_k\Big)&=\widetilde w_2^{2^{k-2}-2^i}\widetilde w_3^{2^j-2^{k-1}}\mathrm{LM}(f_k)\\
&=\widetilde w_2^{2^{k-2}-2^i}\widetilde w_3^{2^j-2^{k-1}}\widetilde w_2^{2^{t-1}-2^k}\widetilde w_3^{2^k-1}\\
&=\widetilde w_2^{2^{t-1}-2^i-2^k+2^{k-2}}\widetilde w_3^{2^j+2^{k-1}-1}\\
&\prec\widetilde w_2^{2^{t-1}-2^i-1},
\end{align*}
for all $k\in\{i+2,\ldots,j+1\}$ such that $f_k\neq0$, we will have that, if $S(f_i,f_j)\neq0$, then (\ref{Sfifj}) is a desired $m$-representation, for $m=\widetilde w_2^{2^{t-1}-2^i-1}\prec\widetilde w_2^{2^{t-1}-2^i}\widetilde w_3^{2^j-1}$. So, we are left to prove (\ref{Sfifj}).

We fix $i$ and prove this by induction on $j$. For $j=i+1$, (\ref{Sfifj}) simplifies to (\ref{eqspol}). Now we suppose that the claim is true for some $j$ such that $i<j\le t-2$, and prove it for $j+1$. By Lemma \ref{lemspol2} and the induction hypothesis we have:
\begin{align*}
S(f_i,f_{j+1})=&\ \widetilde w_3^{2^j}S(f_i,f_j)+\widetilde w_2^{2^{j}-2^i}f_{j+2}\\
=&\ \widetilde w_3^{2^j}\sum_{k=i+2}^{j+1}\widetilde w_2^{2^{k-2}-2^i}\widetilde w_3^{2^j-2^{k-1}}f_k+\widetilde w_2^{2^{j}-2^i}f_{j+2}\\
=&\sum_{k=i+2}^{j+1}\widetilde w_2^{2^{k-2}-2^i}\widetilde w_3^{2^{j+1}-2^{k-1}}f_k+\widetilde w_2^{2^{j}-2^i}f_{j+2}\\
=&\sum_{k=i+2}^{j+2}\widetilde w_2^{2^{k-2}-2^i}\widetilde w_3^{2^{j+1}-2^{k-1}}f_k,
\end{align*}
and the proof of (\ref{Sfifj}) is complete.

Finally, if $i,j\in\{0,1,\ldots,t-1\}$ and $i<j$, then $\mathrm{LM}(f_j)$ cannot divide some monomial in $f_i$ since $2^t-3+2^i<2^t-3+2^j$ ($f_i$ is a homogeneous polynomial in degree $2^t-3+2^i$); and also $\mathrm{LM}(f_i)$ does not divide any of the monomials in $f_j$ since $\mathrm{LM}(f_i)\succ\mathrm{LM}(f_j)$ (Proposition \ref{propvodecimonomi}). So the Gr\"obner basis $\{f_0,f_1,\ldots,f_{t-1}\}$ is the reduced one.
\end{proof}

\subsection{Another Gr\"obner basis}

In \cite[Theorem B(1)]{BasuChakraborty} the authors gave the follo\-wing description of the cohomology algebra $H^*(\widetilde G_{2^t,3})$:
\[H^*(\widetilde G_{2^t,3})\cong\frac{\frac{\mathbb Z_2[\widetilde w_2,\widetilde w_3]}{(g_{2^t-2},g_{2^t-1})}\otimes_{\mathbb Z_2}\mathbb Z_2[a_{2^t-1}]}{(a_{2^t-1}^2-Pa_{2^t-1})},\]
where $a_{2^t-1}$ on the right-hand side corresponds to an "indecomposable element" (denoted by the same symbol) $a_{2^t-1}\in H^{2^t-1}(\widetilde G_{2^t,3})$, and $P$ is some unknown polynomial in $\widetilde w_2$ and $\widetilde w_3$. The algebra on the right-hand side is routinely identified with the quotient $\mathbb Z_2[\widetilde w_2,\widetilde w_3,a_{2^t-1}]/(g_{2^t-2},g_{2^t-1},a_{2^t-1}^2-Pa_{2^t-1})$, and therefore, we have the isomorphism of graded algebras:
\begin{equation}\label{H*}
H^*(\widetilde G_{2^t,3})\cong\frac{\mathbb Z_2[\widetilde w_2,\widetilde w_3,a_{2^t-1}]}{(g_{2^t-2},g_{2^t-1},a_{2^t-1}^2-Pa_{2^t-1})}.
\end{equation}
In the following section we shall prove that the polynomial $P$ vanishes; more precisely, that we can take $P$ to be the zero polynomial (in order for the isomorphism (\ref{H*}) to hold). But for now, we can establish that, whatever the polynomial $P$ is, the set $F:=\{f_0,f_1,\ldots,f_{t-1},a_{2^t-1}^2-Pa_{2^t-1}\}$ is a Gr\"obner basis for the ideal $J:=(g_{2^t-2},g_{2^t-1},a_{2^t-1}^2-Pa_{2^t-1})$ in $\mathbb Z_2[\widetilde w_2,\widetilde w_3,a_{2^t-1}]$.

First of all, let us extend the order on the monomials (defined above) from $\mathbb Z_2[\widetilde w_2,\widetilde w_3]$ to $\mathbb Z_2[\widetilde w_2,\widetilde w_3,a_{2^t-1}]$. We take the lexicographic order with $a_{2^t-1}\succ\widetilde w_2\succ\widetilde w_3$. It is obvious that this is an extension of the previously defined order on $\mathbb Z_2[\widetilde w_2,\widetilde w_3]$, and that
\begin{equation}\label{vodecimonom}
\mathrm{LM}(a_{2^t-1}^2-Pa_{2^t-1})=a_{2^t-1}^2,
\end{equation}
since $P$ is a polynomial in variables $\widetilde w_2$ and $\widetilde w_3$ only.

\begin{theorem}\label{grebnerH*}
The set $F$ is a Gr\"obner basis for $J$.
\end{theorem}
\begin{proof}
The proof relies on Theorem \ref{buchberger}. We already know that
\[J=(g_{2^t-2},g_{2^t-1},a_{2^t-1}^2-Pa_{2^t-1})=(f_0,f_1,\ldots,f_{t-1},a_{2^t-1}^2-Pa_{2^t-1}).\]
By Theorems \ref{grebnerimp*} and \ref{buchberger}, for all $i,j\in\{0,1,\ldots,t-1\}$, $S(f_i,f_j)$ reduces to zero modulo $\{f_0,f_1,\ldots,f_{t-1}\}$. On the other hand, from the definition of the reduction, it is obvious that if a polynomial reduces to zero modulo some subset of $F$, then it reduces to zero modulo $F$. Also, for $i\in\{0,1,\ldots,t-1\}$, from (\ref{vodecimonom}) and Proposition \ref{propvodecimonomi} we see that $\mathrm{LM}(a_{2^t-1}^2-Pa_{2^t-1})$ and $\mathrm{LM}(f_i)$ are relatively prime. According to \cite[Lemma 5.66]{Becker}, $S(f_i,a_{2^t-1}^2-Pa_{2^t-1})$ reduces to zero modulo $\{f_i,a_{2^t-1}^2-Pa_{2^t-1}\}$, and consequently, modulo $F$. By Theorem \ref{buchberger}, $F$ is a Gr\"obner basis for $J$.
\end{proof}

\begin{remark}
We cannot claim that the Gr\"obner basis $F$ is reduced, since the polynomial $P$ is unknown at the moment. However, if we take $P=0$ (in the following section it will be proved that (\ref{H*}) holds in that case), then it is obvious that $F$ is reduced.
\end{remark}

From this theorem, using (\ref{H*}) and Proposition \ref{ekvusl} (implication (1)$\Rightarrow$(3)), we can detect a vector space basis for the cohomology algebra $H^*(\widetilde G_{2^t,3})$. It is the set of all monomials in $a_{2^t-1}$, $\widetilde w_2$ and $\widetilde w_3$ not divisible by any of the leading monomials from $F$. Therefore, by (\ref{vodecimonom}) and Proposition \ref{propvodecimonomi}, we have the following corollary.

\begin{corollary}\label{B}
The set of cohomology classes
\[B=\Big\{a_{2^t-1}^r\widetilde w_2^b\widetilde w_3^c\mid r<2,\, \big(\forall i\in\{0,1,\ldots,t-1\}\big) \, \, b<2^{t-1}-2^i \, \, \vee \, \, c<2^i-1\Big\}\]
is a vector space basis of the cohomology algebra $H^*(\widetilde G_{2^t,3})$.
\end{corollary}

\begin{example}
In the (trivial) case $t=2$ we have that $\widetilde G_{4,3}\approx\widetilde G_{4,1}\approx S^3$. The Gr\"obner basis $F$ is now
\[F=\{f_0,f_1,a_3^2-Pa_3\}=\{g_2,g_3,a_3^2-Pa_3\}=\{\widetilde w_2,\widetilde w_3,a_3^2-Pa_3\}\]
(see Example \ref{exg}), and the additive basis of $H^*(\widetilde G_{4,3})$ is $B=\{1,a_3\}$.
\end{example}

\begin{example}
In the case $t=3$, using again Example \ref{exg}, we have
\[F=\{f_0,f_1,f_2,a_7^2-Pa_7\}=\{g_6,g_7,g_9,a_7^2-Pa_7\}=\{\widetilde w_2^3+\widetilde w_3^2,\widetilde w_2^2\widetilde w_3,\widetilde w_3^3,a_7^2-Pa_7\}\]
The additive basis of $H^*(\widetilde G_{8,3})$ is now (we sort the basis elements by their cohomological dimension)
\[B=\{1,\widetilde w_2,\widetilde w_3,\widetilde w_2^2,\widetilde w_2\widetilde w_3,\widetilde w_3^2,a_7,\widetilde w_2\widetilde w_3^2,a_7\widetilde w_2,a_7\widetilde w_3,a_7\widetilde w_2^2,a_7\widetilde w_2\widetilde w_3,a_7\widetilde w_3^2,a_7\widetilde w_2\widetilde w_3^2\}\]
(cf.\ \cite[Proposition 3.1(3)]{KorbasRusin:Palermo}). For instance,
\[H^7(\widetilde G_{8,3})=\mathbb Z_2\langle a_7\rangle, H^8(\widetilde G_{8,3})=\mathbb Z_2\langle\widetilde w_2\widetilde w_3^2\rangle, H^9(\widetilde G_{8,3})=\mathbb Z_2\langle a_7\widetilde w_2\rangle,\]
\[H^{13}(\widetilde G_{8,3})=\mathbb Z_2\langle a_7\widetilde w_3^2\rangle, H^{14}(\widetilde G_{8,3})=0, H^{15}(\widetilde G_{8,3})=\mathbb Z_2\langle a_7\widetilde w_2\widetilde w_3^2\rangle.\]
\end{example}

\begin{example}
For $t=4$, we work with the Grassmann manifold $\widetilde G_{16,3}$, whose dimension is $(16-3)\cdot3=39$. The leading monomials of polynomials in $F=\{f_0,f_1,f_2,f_3,a_{15}^2-Pa_{15}\}$ are the following (see (\ref{vodecimonom}) and Proposition \ref{propvodecimonomi}):
\[\mathrm{LM}(F)=\{\widetilde w_2^7,\widetilde w_2^6\widetilde w_3,\widetilde w_2^4\widetilde w_3^3,\widetilde w_3^7,a_{15}^2\}.\]
Again sorting the basis elements by their cohomological dimension we get that
\begin{align*}
B=\big\{&1,\widetilde w_2,\widetilde w_3,\widetilde w_2^2,\widetilde w_2\widetilde w_3,\widetilde w_2^3,\widetilde w_3^2,\widetilde w_2^2\widetilde w_3,\widetilde w_2^4,\widetilde w_2\widetilde w_3^2,\widetilde w_2^3\widetilde w_3,\widetilde w_3^3,\widetilde w_2^5,\widetilde w_2^2\widetilde w_3^2,\widetilde w_2^4\widetilde w_3,\widetilde w_2\widetilde w_3^3,\\
&\widetilde w_2^6,\widetilde w_2^3\widetilde w_3^2,\widetilde w_3^4,\widetilde w_2^5\widetilde w_3,\widetilde w_2^2\widetilde w_3^3,\widetilde w_2^4\widetilde w_3^2,\widetilde w_2\widetilde w_3^4,a_{15},\widetilde w_2^3\widetilde w_3^3,\widetilde w_3^5,\widetilde w_2^5\widetilde w_3^2,\widetilde w_2^2\widetilde w_3^4,a_{15}\widetilde w_2,\\
&\widetilde w_2\widetilde w_3^5,a_{15}\widetilde w_3,\widetilde w_2^3\widetilde w_3^4,\widetilde w_3^6,a_{15}\widetilde w_2^2,\widetilde w_2^2\widetilde w_3^5,a_{15}\widetilde w_2\widetilde w_3,\widetilde w_2\widetilde w_3^6,a_{15}\widetilde w_2^3,a_{15}\widetilde w_3^2,\widetilde w_2^3\widetilde w_3^5,\\
&a_{15}\widetilde w_2^2\widetilde w_3,\widetilde w_2^2\widetilde w_3^6,a_{15}\widetilde w_2^4,a_{15}\widetilde w_2\widetilde w_3^2,a_{15}\widetilde w_2^3\widetilde w_3,a_{15}\widetilde w_3^3,\widetilde w_2^3\widetilde w_3^6,a_{15}\widetilde w_2^5,a_{15}\widetilde w_2^2\widetilde w_3^2,\\
&a_{15}\widetilde w_2^4\widetilde w_3,a_{15}\widetilde w_2\widetilde w_3^3,a_{15}\widetilde w_2^6,a_{15}\widetilde w_2^3\widetilde w_3^2,a_{15}\widetilde w_3^4,a_{15}\widetilde w_2^5\widetilde w_3,a_{15}\widetilde w_2^2\widetilde w_3^3,a_{15}\widetilde w_2^4\widetilde w_3^2,\\
&a_{15}\widetilde w_2\widetilde w_3^4,a_{15}\widetilde w_2^3\widetilde w_3^3,a_{15}\widetilde w_3^5,a_{15}\widetilde w_2^5\widetilde w_3^2,a_{15}\widetilde w_2^2\widetilde w_3^4,a_{15}\widetilde w_2\widetilde w_3^5,a_{15}\widetilde w_2^3\widetilde w_3^4,a_{15}\widetilde w_3^6,\\
&a_{15}\widetilde w_2^2\widetilde w_3^5,a_{15}\widetilde w_2\widetilde w_3^6,a_{15}\widetilde w_2^3\widetilde w_3^5,a_{15}\widetilde w_2^2\widetilde w_3^6,a_{15}\widetilde w_2^3\widetilde w_3^6\big\}
\end{align*}
is a vector space basis for $H^*(\widetilde G_{16,3})$. For instance,
\begin{align*}
H^{15}(\widetilde G_{16,3})&=\mathbb Z_2\langle a_{15}\rangle\oplus\mathbb Z_2\langle\widetilde w_2^3\widetilde w_3^3\rangle\oplus\mathbb Z_2\langle\widetilde w_3^5\rangle, H^{16}(\widetilde G_{16,3})=\mathbb Z_2\langle\widetilde w_2^5\widetilde w_3^2\rangle\oplus\mathbb Z_2\langle\widetilde w_2^2\widetilde w_3^4\rangle,\\
H^{17}(\widetilde G_{16,3})&=\mathbb Z_2\langle a_{15}\widetilde w_2\rangle\oplus\mathbb Z_2\langle\widetilde w_2\widetilde w_3^5\rangle, H^{30}(\widetilde G_{16,3})=\mathbb Z_2\langle a_{15}\widetilde w_2^3\widetilde w_3^3\rangle\oplus\mathbb Z_2\langle a_{15}\widetilde w_3^5\rangle,\\
H^{31}(\widetilde G_{16,3})&=\mathbb Z_2\langle a_{15}\widetilde w_2^5\widetilde w_3^2\rangle\oplus\mathbb Z_2\langle a_{15}\widetilde w_2^2\widetilde w_3^4\rangle,
H^{32}(\widetilde G_{16,3})=\mathbb Z_2\langle a_{15}\widetilde w_2\widetilde w_3^5\rangle,\\
H^{37}(\widetilde G_{16,3})&=\mathbb Z_2\langle a_{15}\widetilde w_2^2\widetilde w_3^6\rangle, H^{38}(\widetilde G_{16,3})=0, H^{39}(\widetilde G_{16,3})=\mathbb Z_2\langle a_{15}\widetilde w_2^3\widetilde w_3^6\rangle.
\end{align*}
\end{example}

\subsection{Some calculations in $H^*(\widetilde G_{2^t,3})$}

Recall that the height of a cohomology class $\sigma$ is the maximal integer $m$ with the property $\sigma^m\neq0$.

\begin{lemma}\label{visina}
The height of the class $\widetilde w_3\in H^*(\widetilde G_{2^t,3})$ is $2^{t-1}-2$.
\end{lemma}
\begin{proof}
The polynomial $f_{t-1}$ is a polynomial in the ideal $J=(g_{2^t-2},g_{2^t-1},a_{2^t-1}^2-Pa_{2^t-1})\trianglelefteq\mathbb Z_2[\widetilde w_2,\widetilde w_3,a_{2^t-1}]$, and so $f_{t-1}=0$ in $H^*(\widetilde G_{2^t,3})$ (by (\ref{H*})). On the other hand, according to Proposition \ref{propvodecimonomi}, $f_{t-1}=\widetilde w_3^{2^{t-1}-1}$, so we have that
$\widetilde w_3^{2^{t-1}-1}=0$ in $H^*(\widetilde G_{2^t,3})$.

The class $\widetilde w_3^{2^{t-1}-2}$ is an element of the additive basis $B$ from Corollary \ref{B}, which implies that $\widetilde w_3^{2^{t-1}-2}\neq0$ in $H^*(\widetilde G_{2^t,3})$, completing the proof of the lemma.
\end{proof}

The following proposition concerns the subalgebra $\im p^*$ of the algebra $H^*(\widetilde G_{2^t,3})$. This is the subalgebra generated by the Stiefel--Whitney classes $\widetilde w_2$ and $\widetilde w_3$. Note that the dimension of the manifold $\widetilde G_{2^t,3}$ is $3\cdot(2^t-3)=3\cdot2^t-9$, and so the graded algebra $H^*(\widetilde G_{2^t,3})$ is concentrated in the degrees (cohomological dimensions) from $0$ to $3\cdot2^t-9$. The proposition states that, roughly, in the first third of this range all cohomology classes belong to $\im p^*$, and in the last third there are no nonzero cohomology classes in $\im p^*$.

\begin{proposition}\label{propimp*}
Let $\im p^*$ be the subalgebra of $H^*(\widetilde G_{2^t,3})$ induced by the covering map $p:\widetilde G_{2^t,3}\rightarrow G_{2^t,3}$.
\begin{itemize}
\item[(a)] If $j<2^t-1$, then $H^j(\widetilde G_{2^t,3})\subset\im p^*$.
\item[(b)] If $j>2\cdot2^t-8$, then $H^j(\widetilde G_{2^t,3})\cap\im p^*=0$.
\end{itemize}
\end{proposition}
\begin{proof}
In the additive basis $B$ (from Corollary \ref{B}) the element of the smallest cohomological dimension which is not of the form $\widetilde w_2^b\widetilde w_3^c$ is $a_{2^t-1}$, whose dimension is $2^t-1$. This proves (a).

For (b), assume to the contrary that there is a nonzero class $\sigma\in H^j(\widetilde G_{2^t,3})\cap\im p^*$ for some $j>2\cdot2^t-8$. By the Poincar\'e duality, there exists a class $\tau\in H^{3\cdot2^t-9-j}(\widetilde G_{2^t,3})$ such that $\sigma\tau\neq0$ in $H^{3\cdot2^t-9}(\widetilde G_{2^t,3})$. But $3\cdot2^t-9-j<3\cdot2^t-9-(2\cdot2^t-8)=2^t-1$, and (a) implies that $\tau\in\im p^*$ as well. We conclude that $\sigma\tau$, the nonzero element of $H^{3\cdot2^t-9}(\widetilde G_{2^t,3})$, also belongs to $\im p^*$, which means that $p^*:H^{3\cdot2^t-9}(G_{2^t,3})\rightarrow H^{3\cdot2^t-9}(\widetilde G_{2^t,3})$ is nontrivial. On the other hand, it is well known that $H^{3\cdot2^t-9}(G_{2^t,3})=\mathbb Z_2\langle w_3^{2^t-3}\rangle$, where $w_3=w_3(\gamma_{2^t,3})$ is the third Stiefel--Whitney class of the canonical bundle $\gamma_{2^t,3}$ over $G_{2^t,3}$ (see e.g.\ \cite{Jaworowski}). Finally, we get that $\widetilde w_3^{2^t-3}=p^*(w_3^{2^t-3})\neq0$, contradicting Lemma \ref{visina}.
\end{proof}

\section{Proof of Theorem \ref{thm1}}
\label{proof}

We are going to investigate the action of the Steenrod squares $Sq^1$ and $Sq^2$ on $H^*(\widetilde G_{2^t,3})$. Therefore, it will be convenient to have the formulas from the next lemma at our disposal.

\begin{lemma}\label{Sq1Sq2}
In $H^*(\widetilde G_{2^t,3})$ the following identities hold (for all nonnegative integers $b$ and $c$):
\begin{itemize}
\item[(a)] $Sq^1(\widetilde w_2^b\widetilde w_3^c)=b\widetilde w_2^{b-1}\widetilde w_3^{c+1}$;
\item[(b)] $Sq^2(\widetilde w_2^b\widetilde w_3^c)=(b+c)\widetilde w_2^{b+1}\widetilde w_3^c+\binom{b}{2}\widetilde w_2^{b-2}\widetilde w_3^{c+2}$.
\end{itemize}
\end{lemma}
\begin{proof}
Throughout the proof we will repeatedly use the formulas of Cartan and Wu (see e.g.\ \cite[p.\ 91 and p.\ 94]{MilnorSt}). By the Wu formula:
\[Sq^1(\widetilde w_2)=\widetilde w_3, \quad Sq^1(\widetilde w_3)=0, \quad Sq^2(\widetilde w_2)=\widetilde w_2^2, \quad Sq^2(\widetilde w_3)=\widetilde w_2\widetilde w_3.\]
We prove both formulas by induction on $b$.

(a) Since $Sq^1(\widetilde w_3)=0$, the formula is true for $b=0$. For $b\geq1$, assuming the formula to be true when the exponent of $\widetilde w_2$ is less then $b$, we have
\begin{align*}
Sq^1(\widetilde w_2^b\widetilde w_3^c)&=Sq^1(\widetilde w_2\widetilde w_2^{b-1}\widetilde w_3^c)=Sq^1(\widetilde w_2)\widetilde w_2^{b-1}\widetilde w_3^c+\widetilde w_2Sq^1(\widetilde w_2^{b-1}\widetilde w_3^c)\\
&=\widetilde w_3\widetilde w_2^{b-1}\widetilde w_3^c+(b-1)\widetilde w_2\widetilde w_2^{b-2}\widetilde w_3^{c+1}=b\widetilde w_2^{b-1}\widetilde w_3^{c+1}.
\end{align*}

(b) Suppose first that $b=0$. We want to prove that $Sq^2(\widetilde w_3^c)=c\widetilde w_2\widetilde w_3^c$, which is obviously true for $c=0$ and $c=1$. Proceeding by induction on $c$, we obtain:
\begin{align*}
Sq^2(\widetilde w_3^c)&=Sq^2(\widetilde w_3\widetilde w_3^{c-1})=Sq^2(\widetilde w_3)\widetilde w_3^{c-1}+Sq^1(\widetilde w_3)Sq^1(\widetilde w_3^{c-1})+\widetilde w_3Sq^2(\widetilde w_3^{c-1})\\
&=\widetilde w_2\widetilde w_3\widetilde w_3^{c-1}+(c-1)\widetilde w_3\widetilde w_2\widetilde w_3^{c-1}=c\widetilde w_2\widetilde w_3^c,
\end{align*}
completing the proof for $b=0$.

For the induction step, we have:
\begin{align*}
Sq^2(\widetilde w_2^b\widetilde w_3^c)=&\,Sq^2(\widetilde w_2\widetilde w_2^{b-1}\widetilde w_3^c)\\
=&\,Sq^2(\widetilde w_2)\widetilde w_2^{b-1}\widetilde w_3^c+Sq^1(\widetilde w_2)Sq^1(\widetilde w_2^{b-1}\widetilde w_3^c)+\widetilde w_2Sq^2(\widetilde w_2^{b-1}\widetilde w_3^c)\\
=&\,\widetilde w_2^2\widetilde w_2^{b-1}\widetilde w_3^c+(b-1)\widetilde w_3\widetilde w_2^{b-2}\widetilde w_3^{c+1}\\
&+\widetilde w_2\Big((b-1+c)\widetilde w_2^b\widetilde w_3^c+\mbox{$\binom{b-1}{2}$}\widetilde w_2^{b-3}\widetilde w_3^{c+2}\Big)\\
=&\,(b+c)\widetilde w_2^{b+1}\widetilde w_3^c+\mbox{$\binom{b}{2}$}\widetilde w_2^{b-2}\widetilde w_3^{c+2},
\end{align*}
and we are done.
\end{proof}

Recall that $a_{2^t-1}\in H^{2^t-1}(\widetilde G_{2^t,3})$ is a generator of the algebra $H^*(\widetilde G_{2^t,3})$ (see (\ref{H*})).

\begin{lemma}\label{Sq1Sq2a}
Let $\im p^*$ be the subalgebra of $H^*(\widetilde G_{2^t,3})$ induced by the covering map $p:\widetilde G_{2^t,3}\rightarrow G_{2^t,3}$.
\begin{itemize}
\item[(a)] $Sq^1(a_{2^t-1})\in\im p^*$.
\item[(b)] $Sq^2(a_{2^t-1})\in\im p^*$.
\end{itemize}
\end{lemma}
\begin{proof}
(a) We know that $Sq^1(a_{2^t-1})\in H^{2^t}(\widetilde G_{2^t,3})$, and by Corollary \ref{B}, all basis elements in $H^{2^t}(\widetilde G_{2^t,3})$ are of the form $\widetilde w_2^b\widetilde w_3^c$, i.e., $H^{2^t}(\widetilde G_{2^t,3})\subset\im p^*$.

(b) Since $Sq^2(a_{2^t-1})\in H^{2^t+1}(\widetilde G_{2^t,3})$, $Sq^2(a_{2^t-1})$ is uniquely expressed as a sum of elements of $B$ (from Corollary \ref{B}) with dimension equal to $2^t+1$. Therefore,
\begin{equation}\label{Sq2}
Sq^2(a_{2^t-1})=\beta a_{2^t-1}\widetilde w_2+W,
\end{equation}
for some $\beta\in\mathbb Z_2$ and $W\in\im p^*$. It suffices to show that $\beta=0$.

Observe the squaring operation $Sq^2:H^{3\cdot2^t-11}(\widetilde G_{2^t,3})\rightarrow H^{3\cdot2^t-9}(\widetilde G_{2^t,3})$. Since the dimension of the manifold $\widetilde G_{2^t,3}$ is $3\cdot2^t-9$, this squaring operation is multiplication with the Wu class $v_2\in H^2(\widetilde G_{2^t,3})$. To compute this Wu class, we first use \cite[Theorem 11.14]{MilnorSt}, and conclude that for the second Stiefel--Whitney class of (the tangent bundle over) the manifold $\widetilde G_{2^t,3}$ we have
\[w_2(\widetilde G_{2^t,3})=Sq^2(v_0)+Sq^1(v_1)+v_2=v_2\]
(because $v_1\in H^1(\widetilde G_{2^t,3})=0$). Now, the second Stiefel--Whitney class $w_2(G_{2^t,3})$ of the ("unoriented") Grassmannian $G_{2^t,3}$ vanishes by \cite[Theorem 1.1]{BartikKorbas}. It is well known that the pullback of the tangent bundle over $G_{2^t,3}$ via $p$ is the tangent bundle over $\widetilde G_{2^t,3}$, so we obtain
\[v_2=w_2(\widetilde G_{2^t,3})=p^*w_2(G_{2^t,3})=p^*(0)=0.\]
Therefore, $Sq^2:H^{3\cdot2^t-11}(\widetilde G_{2^t,3})\rightarrow H^{3\cdot2^t-9}(\widetilde G_{2^t,3})$ is the zero map.

By Corollary \ref{B}, (if $t\geq3$) $H^{3\cdot2^t-11}(\widetilde G_{2^t,3})=\mathbb Z_2\langle a_{2^t-1}\widetilde w_2^{2^{t-2}-2}\widetilde w_3^{2^{t-1}-2}\rangle$ and $H^{3\cdot2^t-9}(\widetilde G_{2^t,3})=\mathbb Z_2\langle a_{2^t-1}\widetilde w_2^{2^{t-2}-1}\widetilde w_3^{2^{t-1}-2}\rangle$, and we compute:
\begin{align*}
0=&\,Sq^2(a_{2^t-1}\widetilde w_2^{2^{t-2}-2}\widetilde w_3^{2^{t-1}-2})\\
=&\,Sq^2(a_{2^t-1})\widetilde w_2^{2^{t-2}-2}\widetilde w_3^{2^{t-1}-2}+Sq^1(a_{2^t-1})Sq^1(\widetilde w_2^{2^{t-2}-2}\widetilde w_3^{2^{t-1}-2})\\
&+a_{2^t-1}Sq^2(\widetilde w_2^{2^{t-2}-2}\widetilde w_3^{2^{t-1}-2}).
\end{align*}
Since Steenrod squares are cohomology operations, they map $\im p^*$ to $\im p^*$, and we conclude that the second summand is zero by the part (a) of this lemma and Proposition \ref{propimp*}(b). Continuing the calculation, by (\ref{Sq2}) and Lemma \ref{Sq1Sq2}(b) we get
\begin{align*}
0=&\,(\beta a_{2^t-1}\widetilde w_2+W)\widetilde w_2^{2^{t-2}-2}\widetilde w_3^{2^{t-1}-2}\\
&+a_{2^t-1}\Big((2^{t-2}-2+2^{t-1}-2)\widetilde w_2^{2^{t-2}-1}\widetilde w_3^{2^{t-1}-2}+\mbox{$\binom{2^{t-2}-2}{2}$}\widetilde w_2^{2^{t-2}-4}\widetilde w_3^{2^{t-1}}\Big).
\end{align*}
For the same reason as above, $W\cdot\widetilde w_2^{2^{t-2}-2}\widetilde w_3^{2^{t-1}-2}=0$, and $\widetilde w_3^{2^{t-1}}=0$ by Lemma \ref{visina}. This means that
\[0=\beta a_{2^t-1}\widetilde w_2^{2^{t-2}-1}\widetilde w_3^{2^{t-1}-2} \quad\mbox{ in } H^{3\cdot2^t-9}(\widetilde G_{2^t,3})=\mathbb Z_2\langle a_{2^t-1}\widetilde w_2^{2^{t-2}-1}\widetilde w_3^{2^{t-1}-2}\rangle,\]
implying $\beta=0$.
\end{proof}

Finally, we can prove our main result.

\noindent{\bf Proof of Theorem \ref{thm1}.} By (\ref{H*}) we know that the cohomology algebra $H^*(\widetilde G_{2^t,3})$ is isomorphic to the quotient $\mathbb Z_2[\widetilde w_2,\widetilde w_3,a_{2^t-1}]/(g_{2^t-2},g_{2^t-1},a_{2^t-1}^2-Pa_{2^t-1})$. Here, $P$ is any polynomial in $\widetilde w_2$ and $\widetilde w_3$ with the property $a_{2^t-1}^2=Pa_{2^t-1}$ in $H^*(\widetilde G_{2^t,3})$. If we show that $a_{2^t-1}^2=0$ in $H^*(\widetilde G_{2^t,3})$, we will have
\begin{align*}
H^*(\widetilde G_{2^t,3})&\cong\frac{\mathbb Z_2[\widetilde w_2,\widetilde w_3,a_{2^t-1}]}{(g_{2^t-2},g_{2^t-1},a_{2^t-1}^2)}\cong\frac{\mathbb Z_2[\widetilde w_2,\widetilde w_3]}{(g_{2^t-2},g_{2^t-1})}\otimes_{\mathbb Z_2}\frac{\mathbb Z_2[a_{2^t-1}]}{(a_{2^t-1}^2)}\\
&=\frac{\mathbb Z_2[\widetilde w_2,\widetilde w_3]}{(g_{2^t-2},g_{2^t-1})}\otimes_{\mathbb Z_2}\Lambda_{\mathbb Z_2}(a_{2^t-1}),
\end{align*}
and this is what we are supposed to prove. So, it suffices to show that $a_{2^t-1}^2=0$ in $H^*(\widetilde G_{2^t,3})$.

By Corollary \ref{B}, the cohomology class $a_{2^t-1}^2\in H^{2\cdot 2^t-2}(\widetilde G_{2^t,3})$ is a linear combination of the classes
\[a_{2^t-1}\widetilde w_2^{2^{t-1}-5}\widetilde w_3^3,\, a_{2^t-1}\widetilde w_2^{2^{t-1}-8}\widetilde w_3^5, \ldots ,\, a_{2^t-1}\widetilde w_2^{2^{t-1}-2-3k}\widetilde w_3^{2k+1}, \ldots\]
and so on, as long as $2^{t-1}-2-3k\geq0$, i.e., $k\leq\frac{2^{t-1}-2}{3}$. Namely, the cohomological dimension of all of these classes is $2\cdot 2^t-2$, and all of them belong to $B$. Indeed, if $2^{t-1}-2-3k\geq2^{t-1}-2^i$ and $2k+1\geq2^i-1$ for some $i\in\{0,1,\ldots,t-1\}$, then $2^i\geq3k+2$ and $2^i\leq2k+2$, and this is not possible for $k\geq1$ (the monomial $a_{2^t-1}\widetilde w_2^{2^{t-1}-2}\widetilde w_3$ also belongs to $H^{2\cdot 2^t-2}(\widetilde G_{2^t,3})$, but it is not in $B$ -- this monomial is divisible by the leading monomial of $f_1\in F$). Therefore,
\[a_{2^t-1}^2=\sum_{k=1}^{\lfloor\frac{2^{t-1}-2}{3}\rfloor}\lambda_ka_{2^t-1}\widetilde w_2^{2^{t-1}-2-3k}\widetilde w_3^{2k+1},\]
for some (uniquely determined) $\lambda_k\in\mathbb Z_2$. We want to prove that $\lambda_k=0$ for all $k$.

By the Cartan formula, $Sq^1(a_{2^t-1}^2)=0$, and this is an equality in $H^{2\cdot 2^t-1}(\widetilde G_{2^t,3})$. Since $2\cdot 2^t-1>2\cdot 2^t-8$, Proposition \ref{propimp*}(b) tells us that every class in $H^{2\cdot 2^t-1}(\widetilde G_{2^t,3})\cap\im p^*$ vanishes. Now we use Lemmas \ref{Sq1Sq2a}(a) and \ref{Sq1Sq2}(a) to calculate:
\begin{align*}
0&=Sq^1(a_{2^t-1}^2)=\sum_{k=1}^{\lfloor\frac{2^{t-1}-2}{3}\rfloor}\lambda_kSq^1(a_{2^t-1}\widetilde w_2^{2^{t-1}-2-3k}\widetilde w_3^{2k+1})\\
&=\sum_{k=1}^{\lfloor\frac{2^{t-1}-2}{3}\rfloor}\lambda_k\Big(Sq^1(a_{2^t-1})\widetilde w_2^{2^{t-1}-2-3k}\widetilde w_3^{2k+1}+a_{2^t-1}Sq^1(\widetilde w_2^{2^{t-1}-2-3k}\widetilde w_3^{2k+1})\Big)\\
&=\sum_{k=1}^{\lfloor\frac{2^{t-1}-2}{3}\rfloor}\lambda_k(2^{t-1}-2-3k)a_{2^t-1}\widetilde w_2^{2^{t-1}-3-3k}\widetilde w_3^{2k+2}.\\
\end{align*}
Since $k\geq1$, there is no $i\in\{0,1,\ldots,t-1\}$ with the properties $2^{t-1}-3-3k\geq2^{t-1}-2^i$ and $2k+2\geq2^i-1$, i.e., $2^i\geq3k+3$ and $2^i\leq2k+3$. This means that $a_{2^t-1}\widetilde w_2^{2^{t-1}-3-3k}\widetilde w_3^{2k+2}\in B$ for all $k$, and since these are linearly independent (by Corollary \ref{B}), we can conclude that if $k$ is odd, then $\lambda_k=0$.

Now (substituting $2j$ for $k$) we have
\[a_{2^t-1}^2=\sum_{j=1}^{\lfloor\frac{2^{t-2}-1}{3}\rfloor}\lambda_{2j}a_{2^t-1}\widetilde w_2^{2^{t-1}-2-6j}\widetilde w_3^{4j+1}.\]
Again by the Cartan formula,
\[Sq^2(a_{2^t-1}^2)=\big(Sq^1(a_{2^t-1})\big)^2\in H^{2\cdot 2^t}(\widetilde G_{2^t,3})\cap\im p^*=0\]
(Lemma \ref{Sq1Sq2a}(a) and Proposition \ref{propimp*}(b)), so according to Lemmas \ref{Sq1Sq2} and \ref{Sq1Sq2a}:
\begin{align*}
0&=Sq^2(a_{2^t-1}^2)=\sum_{j=1}^{\lfloor\frac{2^{t-2}-1}{3}\rfloor}\lambda_{2j}Sq^2(a_{2^t-1}\widetilde w_2^{2^{t-1}-2-6j}\widetilde w_3^{4j+1})\\
&=\sum_{j=1}^{\lfloor\frac{2^{t-2}-1}{3}\rfloor}\lambda_{2j}a_{2^t-1}Sq^2(\widetilde w_2^{2^{t-1}-2-6j}\widetilde w_3^{4j+1})\\
&=\sum_{j=1}^{\lfloor\frac{2^{t-2}-1}{3}\rfloor}\lambda_{2j}a_{2^t-1}\Big(\widetilde w_2^{2^{t-1}-1-6j}\widetilde w_3^{4j+1}+\mbox{$\binom{2^{t-1}-2-6j}{2}$}\widetilde w_2^{2^{t-1}-4-6j}\widetilde w_3^{4j+3}\Big)\\
&=\sum_{j=1}^{\lfloor\frac{2^{t-2}-1}{3}\rfloor}\Big(\lambda_{2j}a_{2^t-1}\widetilde w_2^{2^{t-1}-1-6j}\widetilde w_3^{4j+1}+\lambda_{2j}\mbox{$\binom{2^{t-1}-2-6j}{2}$}a_{2^t-1}\widetilde w_2^{2^{t-1}-4-6j}\widetilde w_3^{4j+3}\Big).\\
\end{align*}
The monomials $a_{2^t-1}\widetilde w_2^{2^{t-1}-1-6j}\widetilde w_3^{4j+1}$ and $a_{2^t-1}\widetilde w_2^{2^{t-1}-4-6j}\widetilde w_3^{4j+3}$, for various $j$, are pairwise different elements of the additive basis $B$ (using the fact $j\geq1$, one can check this in a similar way as above), and finally, we obtain that $\lambda_{2j}=0$ for all $j$.
\qed

\bibliographystyle{amsplain}

\end{document}